\documentclass[11pt]{amsart}

\usepackage{amsmath,amssymb,amsthm}
\usepackage{geometry}
\usepackage{hyperref}
\usepackage[numbers,sort&compress]{natbib}

\newtheorem{theorem}{Theorem}[section]
\newtheorem{lemma}[theorem]{Lemma}
\newtheorem{proposition}[theorem]{Proposition}

\theoremstyle{definition}
\newtheorem{definition}[theorem]{Definition}
\newtheorem{example}[theorem]{Example}
\newtheorem{conjecture}[theorem]{Conjecture}
\theoremstyle{remark}

\numberwithin{equation}{section}

\title{Table Problem Revisited}

\author{Xiao-Song Yang}
\address{School of Mathematics and Statistics,
Huazhong University of Science and Technology,
Wuhan, 430074, China}
\email{yangxs@hust.edu.cn}


\date{\today}

\begin{document}

\begin{abstract}
  We revisit Fenn's table theorem from a differential-topological point
of view.  We prove a zero-existence theorem on a cylinder, which gives
a short proof of the horizontal square-table theorem under Fenn's
boundary conditions, and We establish a theorem under more general boundary conditions. 
 We also discuss square tables on saddle surfaces and
conjecture that every sufficiently small square table can be placed 
horizontally on a saddle surface.  We further conjecture that any 
prescribed rectangular table can be placed horizontally on a Fenn graph.
\end{abstract}

\maketitle

%======================================================================
\section{Introduction}
%======================================================================

Let $V=[v_{1},v_{2},v_{3},v_{4}]\subset\mathbb{R}^{3}$ be a set of
four points that forms a square with $\overline{v_{1}v_{3}}$ and
$\overline{v_{2}v_{4}}$ being the diagonals.  Let $G(f)$ be the
graph defined by a continuous function $f:\mathbb{R}^{2}\to\mathbb{R}$,
that is
$G(f)=\{(x,f(x))\in\mathbb{R}^{3}: x\in\mathbb{R}^{2}\}$.
The well-known table problem can be mathematically formulated as
follows.

Given a configuration $V=[v_{1},v_{2},v_{3},v_{4}]\subset\mathbb{R}^{3}$,
is it possible to find a transformation (translation and rotation)
$T:\mathbb{R}^{3}\to\mathbb{R}^{3}$ such that
$T(V)\subset G(f)$?  Furthermore, is it possible to have
$f(Tv_{1})=f(Tv_{2})=f(Tv_{3})=f(Tv_{4})$?

A related question that remains open is whether an affirmative
answer can be given when $V$ is the set of vertices of a rectangle.

Inspired by Dyson's theorem on real-valued continuous maps on the
$2$-dimensional unit sphere $S^{2}$, R.~Fenn proved the following
theorem, which has an intuitive explanation for horizontally
balancing a square table on an uneven floor or hills.

\begin{theorem}[Fenn's theorem]\label{thm:fenn}
  Let $D\subset\mathbb{R}^{2}$ be a convex compact set.  Let
  $f:D\to\mathbb{R}$ be a continuous map satisfying the following
  conditions:
  \begin{enumerate}
    \item[(i)] $f(x)=0$ if $x\in\mathbb{R}^{2}-D$;
    \item[(ii)] $f(x)\ge 0$ if $x\in D$.
  \end{enumerate}
  Then for any $a>0$, there is a square in the plane
  $\mathbb{R}^{2}$ centered at a point $p\in D$, with side length
  $a$, such that $f$ takes the same value on the four vertices of
  this square.
\end{theorem}

Fenn's theorem depends on the specific boundary conditions on the
continuous map.

As in \cite{Fenn70}, let $D\subset\mathbb{R}^{3}$ be a compact
convex domain of the $x_{1}x_{2}$-plane with the origin in its
interior.  Denote by $a_{1},a_{2},a_{3}$ and $a_{4}$ the vertices
of a square in the plane with center $p$ and fixed side length $d$.
Let $\Pi(p,\alpha)$ be the plane determined by the points
$(a_{1},f(a_{1}))$, $(a_{2},f(a_{2}))$ and $(a_{3},f(a_{3}))$,
with $\alpha$ being the angle between the $x_{1}$-axis and the
vector $\overline{a_{1}a_{2}}$.  Denote by $L(a_{4})$ the line
parallel to the $x_{3}$-axis passing through $a_{4}$.  Let
$q(p,\alpha)=\Pi(p,\alpha)\cap L(a_{4})$.  Then the function
$\phi:D\times S^{1}\to\mathbb{R}$ defined by
$\phi(p,\alpha)=q-f(a_{4})$ is clearly continuous because the
fourth vertex $a_{4}$ depends continuously on the center point $p$
and the angle $\alpha$.

The function $\phi$ has the following properties.
\begin{itemize}
  \item[(i)] If $\phi(x,\alpha)=0$, then
    $\phi(x,\alpha+\frac{\pi}{2})=0$.
  \item[(ii)] If $\phi(x,\alpha)\neq0$, then
    $\phi(x,\alpha)\cdot\phi(x,\alpha+\frac{\pi}{2})<0$.
\end{itemize}

For $(x,\alpha)\in D\times S^{1}$, let $n(x,\alpha)$ be the unit
vector normal to the plane $\Pi(x,\alpha)$ with the
$x_{3}$-component being positive.  Let $v(x,\alpha)=P\circ
n(x,\alpha)$, where $P:\mathbb{R}^{3}\to\mathbb{R}^{2}$ is the
projection map $P(x_{1},x_{2},x_{3})=(x_{1},x_{2},0)$.  We obtain
a map $\Phi:D\times S^{1}\to\mathbb{R}^{3}$ as follows
\begin{equation}\label{eq:Phi}
  \Phi(x,\alpha)=
  \begin{pmatrix}
    \phi(x,\alpha)\\[2pt]
    \psi(x,\alpha)
  \end{pmatrix},
\end{equation}
with $\psi(x,\alpha)=v(x,\alpha)+v(x,\alpha+\frac{\pi}{2})
  +v(x,\alpha+\pi)+v(x,\alpha+\frac{3\pi}{2})$.

Note that $\psi(x,\alpha)=\psi(x,\alpha+\frac{\pi}{2})$.

To prove Fenn's theorem, it is clearly enough to prove the
existence of zero points of the map $\Phi$.  For points
$x\in\phi^{-1}(0)$, $\psi(x,\alpha)=4v(x,\alpha)$.

Unlike \cite{Fenn70}, where the arguments are based on homological
techniques in algebraic topology, we present an elementary proof in
terms of differential topology for this elegant theorem.  In
addition, we can generalize this statement to rectangular tables
and further obtain some new interesting results under different
boundary conditions.

This paper is organized as follows.  In Section~2 we prove a
zero-existence theorem for two-dimensional maps on a cylinder.  In
Section~3 we use this result to give a new proof of Fenn's theorem and
to establish a variant under different boundary conditions.  In
Section~4 we study square tables on saddle surfaces and formulate a
related conjecture.  Finally, in Section~5 we discuss rectangular
tables on Fenn graphs and propose a fixed-size rectangular conjecture.

%======================================================================
\section{A Basic Theorem}
%======================================================================

In this section we mainly establish a general result for later
arguments.

Let $D\subset\mathbb{R}^{2}$ be a convex compact set with
$\partial D$ being $C^{1}$ differentiable.  We first prove the
following fact.

\begin{lemma}\label{lem:basic}
  Assume a smooth map $f:D\times[a,b]\to\mathbb{R}^{2}$ satisfies
  the following conditions:
  \begin{itemize}
    \item[$A_{1}$] $f$ is generic and satisfies
      $f(x,s)\neq0$ for all $(x,s)\in D\times[a,b]$.
    \item[$A_{2}$] The origin $o\in\mathbb{R}^{2}$ is a regular
      value of $f$.
    \item[$A_{3}$] $\deg f|_{\partial D\times\{a\}}$ is an odd
      number.
  \end{itemize}
  Then $f^{-1}(0)$ has a component connecting
  $D\times\{a\}$ and $D\times\{b\}$.
\end{lemma}

\begin{proof}
  Condition $A_{3}$ implies that $f$ has zero points in
  $D\times\{a\}$.  Since $f$ is generic, the local index of any
  zero point of $f$ restricted to $D\times\{a\}$ is $-1$ or $1$.
  The oddness of $\deg f|_{\partial D\times\{a\}}$ implies that the
  numbers of zero points of $f$ on $D\times\{a\}$ and
  $D\times\{b\}$ are odd, respectively, by virtue of the well-known
  Poincar\'e--Hopf index theorem.

  Since $0\in\mathbb{R}^{2}$ is a regular value of $f$,
  $f^{-1}(0)$ consists of closed curves and segments with endpoints
  in $D\times\{a\}$ and $D\times\{b\}$.  Suppose that no segment
  connects $D\times\{a\}$ and $D\times\{b\}$, then the numbers of
  zero points of $f$ in $D\times\{a\}$ and $D\times\{b\}$,
  respectively, must be even, because each zero point in
  $D\times\{a\}$ or $D\times\{b\}$ is the endpoint of some segment.
  Therefore the oddness of the number of zero points on
  $D\times\{a\}$ and $D\times\{b\}$ implies the existence of a
  component of $f^{-1}(0)$ that connects $D\times\{a\}$ and
  $D\times\{b\}$.
\end{proof}

Let $D\subset\mathbb{R}^{2}$ be a compact convex domain with
$C^{1}$ boundary.  Related to Fenn's theorem on the table problem,
the main theorem in this section can be stated as follows.

\begin{theorem}\label{thm:main}
  Let $f:D\times[a,b]\to\mathbb{R}^{2}$ be a continuous map.
  Assume that this map satisfies the following conditions:
  \begin{enumerate}
    \item[(1)] $f(x)\cdot u(x)>0$ (or $<0$) for all
      $x\in\partial D\times[a,b]$, where $u(x)$ is the unit vector
      field normal to the boundary $\partial D\times[a,b]$ pointing
      outward.
    \item[(2)] $f(x,a)=f(x,b)$ for all $x\in D$.
  \end{enumerate}
  Let $g:D\times[a,b]\to\mathbb{R}$ be a continuous map satisfying
  the following conditions:
  \begin{enumerate}
    \item[(3)] $g(x,a)=0$ for some $(x,a)\in D\times\{a\}$ implies
      $g(x,b)=0$.
    \item[(4)] $g(x,a)\neq0$ implies
      $g(x,a)\cdot g(x,b)<0$.
  \end{enumerate}
  Then there exists a point
  $(\bar{x},\bar{s})\in D\times[a,b]$ such that
  $f(\bar{x},\bar{s})=0$ and $g(\bar{x},\bar{s})=0$.
\end{theorem}

\begin{proof}
  First, we prove the assertion in the above theorem in the case
  where $f$ is a smooth and generic map.

  Let $D_{0}\subset D$ be the set of points satisfying
  $g(x,a)=0$, then $g(x,b)=0$ by condition~(3).

  Let $D_{1}\subset D$ be the set of points with $g(x,a)>0$, and
  $D_{2}\subset D$ be the set of points with $g(x,a)<0$; then
  $g(x,b)<0$ on $D_{1}\times\{b\}$ and $g(x,b)>0$ on
  $D_{2}\times\{b\}$.

  Since $f(x)\cdot u(x)>0$ (or $<0$) by condition~(1), it is easy
  to see that $\deg f|_{D\times\{a\}}$ is $1$.  Thus $f$ has an odd
  number of zero points in $D\times\{a\}$ by Lemma~\ref{lem:basic}.
  Let $(x_{1},a),\dots,(x_{2k+1},a)$ be the zero points of $f$ on
  $D\times\{a\}$; then by condition~(2) it is easy to see that
  $(x_{1},b),\dots,(x_{2k+1},b)$ are the only zero points of $f$ on
  $D\times\{b\}$.  If one of these points lies in
  $D_{0}\times\{a\}\cup D_{0}\times\{b\}$, then the statement of
  the theorem is clearly true.

  Now suppose that this is not the case.  Without loss of
  generality, we can suppose that an even number of points among
  $(x_{1},a),\dots,(x_{2k+1},a)$, say,
  $Q_{1}=\{(x_{1},a),\dots,(x_{2h},a)\}$, lie in
  $D_{1}^{+}\times\{a\}=D_{1}\times\{a\}$, and an odd number of
  points, say,
  $Q_{2}=\{(x_{2h+1},a),\dots,(x_{2k+1},a)\}$, lie in
  $D_{2}^{-}\times\{a\}=D_{2}\times\{a\}$.
  In view of condition~(4), it follows that
  \begin{itemize}
    \item $P_{1}=\{(x_{1},b),\dots,(x_{2h},b)\}\subset
      D_{1}^{-}=D_{1}\times\{b\}$,
  \end{itemize}
  and
  \begin{itemize}
    \item $P_{2}=\{(x_{2h+1},b),\dots,(x_{2k+1},b)\}\subset
      D_{2}^{+}=D_{2}\times\{b\}$.
  \end{itemize}

  Suppose that $f^{-1}(0)$ has no segment connecting a point in
  $Q_{1}$ and a point in $P_{1}\subset D_{1}^{-}$, then there must
  be an even number of points of $Q_{1}$ that can be connected by
  segments of $f^{-1}(0)$ with the points in
  $P_{2}\subset D_{2}^{+}$.  It follows that only an even number of
  points of $P_{2}$ can be connected with some points of $Q_{1}$.
  Since the number of zero points of $f$ in $P_{2}$ is odd by the
  above supposition, it is easy to see that there is at least one
  point $p\in P_{2}$ that is connected by some segment $l$ of
  $f^{-1}(0)$ with a point $q\in Q_{2}$.  Notice that $g(p)>0$ and
  $g(q)<0$, it is easy to see that on the segment $l$ we have a
  point $(\bar{x},\bar{s})\in l$ such that
  $g(\bar{x},\bar{s})=0$.  Keeping in mind that
  $f(\bar{x},\bar{s})=0$, we see that the assertion holds when $f$
  is smooth and generic.

  Now let us return to the continuous case.  Suppose, on the
  contrary, that the statement of the theorem is false; then there
  exists a positive number $\delta$ such that
  \[
    \|f(x,s)\|+|g(x,s)|\ge\delta
  \]
  for all $(x,s)\in D\times[a,b]$, due to the compactness of
  $D\times[a,b]$.  Let
  $\bar{f}:D\times[a,b]\to\mathbb{R}^{2}$ be a smooth and generic
  map satisfying
  \[
    \|f(x,s)-\bar{f}(x,s)\|\le\frac{\delta}{4}
  \]
  for all $(x,s)\in D\times[a,b]$.  Then from the arguments for the
  smooth case, we have a point
  $(\bar{x},\bar{s})\in D\times[a,b]$ with
  $\bar{f}(\bar{x},\bar{s})=0$ and $g(\bar{x},\bar{s})=0$, and this
  leads to the inequality
  \[
    \|f(\bar{x},\bar{s})\|+|g(\bar{x},\bar{s})\|
    \le\frac{\delta}{4},
  \]
  leading to a contradiction.
\end{proof}

%======================================================================
\section{Table Problem}
%======================================================================

In this section we first give a new proof of Fenn's theorem on the
square table problem.  We then
prove a variant under different boundary conditions.

\subsection{A new proof of Fenn's theorem}

\begin{proof}
  Unlike the arguments in \cite{Fenn70}, we instead consider the
  space $D\times\bigl[0,\frac{\pi}{2}\bigr]$ and the maps
  $\phi:D\times\bigl[0,\frac{\pi}{2}\bigr]\to\mathbb{R}$ and
  $v:D\times\bigl[0,\frac{\pi}{2}\bigr]\to\mathbb{R}^{2}$ as in
  Fenn's theorem reviewed in the first section, i.e., the map
  defined in~\eqref{eq:Phi}.

  It is enough to prove that there is a point
  $(x,s)\in D\times\bigl[0,\frac{\pi}{2}\bigr]$ such that
  $\phi(x,s)=0$ and $v(x,s)=0$.

  Suppose that this is not the case.  Then, because of the
  compactness of the set
  $A=\{(x,s)\in D\times\bigl[0,\frac{\pi}{2}\bigr]:
    \phi(x,s)=0\}$, we have $\|v(x,s)\|\ge\delta$ with $\delta>0$
  on the set $A$.

  Because of the convexity of the domain $D$, it is easy to prove
  that $v(x,s)\cdot u(x)\ge0$ for each $x\in\partial D$, where
  $u(x)$ is the unit vector field normal to the boundary and
  pointing outward.  Now it is easy to see that there exists a
  generic smooth map
  $g:D\times\bigl[0,\frac{\pi}{2}\bigr]\to\mathbb{R}^{2}$ such that
  $g(x,s)\cdot u(x)>0$ for each $x\in\partial D$ and satisfies
  \[
    \|g(x,s)-v(x,s)\|<\frac{\delta}{2}
  \]
  for all $(x,s)\in D\times\bigl[0,\frac{\pi}{2}\bigr]$.  Then by
  Theorem~\ref{thm:main}, there exists a point
  $(\bar{x},\bar{s})\in A$ such that $g(\bar{x},\bar{s})=0$.
  Therefore
  \[
    \|v(\bar{x},\bar{s})\|<\frac{\delta}{2},
  \]
  leading to a contradiction.
\end{proof}

The following new result concerns the boundary conditions.

\begin{theorem}\label{thm:newbc}
  Assume a continuous map $f:D\to\mathbb{R}$ satisfies the
  following conditions:
  \begin{enumerate}
    \item[(i)] $f(x)\ge0$ for $x\in\mathbb{R}^{2}-D$, and there is
      a $\delta>0$ such that $f(x)\le0$ if
      $\operatorname{dist}(x,\partial D)\le\delta$;
  \end{enumerate}
  or
  \begin{enumerate}
    \item[(ii)] $f(x)\le0$ for $x\in\mathbb{R}^{2}-D$, and there is
      a $\delta>0$ such that $f(x)\ge0$ if
      $\operatorname{dist}(x,\partial D)\le\delta$.
  \end{enumerate}
  Then for any $a<\frac{\sqrt{2}}{2}\,\delta$ and a point $p\in D$,
  there is a square in $\mathbb{R}^{2}$ centered at $p$ with side
  length $a$ such that $f$ takes the same value on the four
  vertices of the square.
\end{theorem}

\begin{proof}
  It is enough to consider case~(i).  In this case, we consider the
  space $D\times\bigl[0,\frac{\pi}{2}\bigr]$ and the maps
  $\phi:D\times\bigl[0,\frac{\pi}{2}\bigr]\to\mathbb{R}$ and
  $v:D\times\bigl[0,\frac{\pi}{2}\bigr]\to\mathbb{R}^{2}$, where
  the map $\phi$ and the map $v$ are defined by the vertices of the
  square as in Fenn's theorem, but with side length $a$.  Since the
  domain $D$ is convex, it is easy to prove that
  $v(x,s)\cdot u(x)\le0$ for each $x\in\partial D$, where $u(x)$ is
  the unit vector field normal to the boundary and pointing
  outward.  Then, as in the new proof of Fenn's theorem, we can
  apply Theorem~\ref{thm:main} to complete the proof.
\end{proof}

%======================================================================
\section{Table Problem on Saddle Surfaces}
%======================================================================

In this section we first present a sufficient condition for a
square table to be placed horizontally on the surface defined by a
general graph $G(f)$.  For a continuous function
$f:\mathbb{R}^{2}\to\mathbb{R}$, in view of Theorem~\ref{thm:main} we have the following fact.

\begin{theorem}\label{thm:surface}
  Let $D$ be a domain in $\mathbb{R}^{2}$ with its boundary
  $\partial D$ homeomorphic to the unit circle $S^{1}$.  Let
  $\psi_{r}(x,\alpha)$ be the map defined
  in~\eqref{eq:Phi} with $r$ being the side length of the
  corresponding square.  Assume the map $\psi(x,\alpha)$ restricted
  to the boundary $\partial D$ has odd degree; then one can place
  horizontally a square table with side length $r$ on the graph
  $G(f)$, and the table is centered at a point in $D$.
\end{theorem}

We are interested in table problems on some kinds of surfaces
defined by a special class of continuous functions
$f:\mathbb{R}^{2}\to\mathbb{R}$, the saddle surfaces as defined
below.

\begin{definition}\label{def:saddle}
  Consider a continuous map
  $f:\mathbb{R}^{2}\to\mathbb{R}$.  The graph of $f$ in
  $\mathbb{R}^{3}$ is called a \emph{saddle surface} with center
  $p=(p_{1},p_{2})$, if $f$ satisfies the following conditions:
  \begin{enumerate}
    \item[(i)] $f(x_{1},x_{2})>f(\bar{x}_{1},x_{2})$, provided
      $|x_{1}-p_{1}|>|\bar{x}_{1}-p_{1}|$;
    \item[(ii)] $f(x_{1},x_{2})<f(x_{1},\bar{x}_{2})$, provided
      $|x_{2}-p_{2}|>|\bar{x}_{2}-p_{2}|$.
  \end{enumerate}
\end{definition}

Without loss of generality, we only consider saddle surfaces with
center at the origin.

Here are some examples of saddle surfaces.
\begin{example}\label{ex:saddle1}
  The surface defined by $x_{3}=f(x_{1},x_{2})=ax_{1}^{2}-bx_{2}^{2}$.
\end{example}

\begin{example}\label{ex:saddle2}
  The surface defined by $x_{3}=f(x_{1},x_{2})=g(x_{1})-h(x_{2})$,
  where $g(x)$ and $h(x)$ are even functions and are increasing on
  $[0,\infty)$.
\end{example}

We are interested in the following questions.

\noindent\textbf{Q1.} Can we balance horizontally an equilateral
triangle table on the surface defined by
$x_{3}=f(x_{1},x_{2})=ax_{1}^{2}-bx_{2}^{2}$?

\noindent\textbf{Q2.} Can we balance horizontally a rectangular
table on the surface defined by
$x_{3}=f(x_{1},x_{2})=g(x_{1})-h(x_{2})$ as given in
Example~\ref{ex:saddle2}?

More generally, can we balance horizontally an equilateral triangle
table or a rectangular table on a saddle surface defined in
Definition~\ref{def:saddle}?

The following fact is easy to prove.
\begin{proposition}\label{prop:saddle}
  On the saddle surface defined by Example~\ref{ex:saddle2}, we can
  balance horizontally a square table of any size.
\end{proposition}

This suggests the following conjecture for
saddle surfaces in the sense of Definition~\ref{def:saddle}.
\begin{conjecture}\label{thm:saddle}
  Let $f:\mathbb{R}^{2}\to\mathbb{R}$ be a continuous map such that
  $G(f)$ is a saddle surface with its center being the origin.
  Let $Q\subset\mathbb{R}^{2}$ be the square centered at the origin
  with its sides parallel to the coordinate axes.  Then for any
  square table $S$ with its side length less than half the side
  length of $Q$, there is a point $C\in Q$ such that one can place
  horizontally the table $S$ on the saddle surface and make it
  centered at $C$.
\end{conjecture}

%======================================================================
\section{Balancing Horizontally a Rectangular Table}
%======================================================================

For the convenience of the arguments in the sequel, we give some
definitions.  By a \emph{cyclic} configuration of a four-point set,
we mean that the set
$V=[v_{1},v_{2},v_{3},v_{4}]\subset\mathbb{R}^{3}$ is contained in
a round circle.

\begin{definition}
  Given a continuous map $f:\mathbb{R}^{2}\to\mathbb{R}$, the graph
  $G(f)$ is called a \emph{Fenn graph} if the function $f$
  satisfies the conditions in Fenn's theorem for some compact
  convex domain $D\subset\mathbb{R}^{2}$.
\end{definition}

A typical configuration besides the square is the rectangular
configuration.  Up to now, no solid theory for balancing
horizontally a rectangle on a Fenn graph has appeared in the
literature.  A well-known fact is that any rectangle can be
balanced on every continuous graph $G(f)$ (which means that the
table does not wobble) due to a topological theorem by
Livesay~\cite{Livesay54}, which stated that for every
continuous function $F:S^{2}\to\mathbb{R}$ and every angle
$\theta\in(0,\pi)$, there exist two diameters of $S^{2}$ making
angle $\theta$ such that $F$ takes the same value at the four
endpoints of these two diameters.  Since the vertices of a
rectangle inscribed in a circle are exactly the endpoints of two
diameters making a fixed angle, Livesay's theorem implies that
any prescribed rectangle can be positioned on an arbitrary
continuous graph so that its four legs touch the graph.  In other
words, the table can be balanced in the sense that it does not
wobble.  However, this does not mean that the table top is
horizontal, since the common plane determined by the four contact
points need not be parallel to the $x_{1}x_{2}$-plane.

Nonetheless, many evidences show that one can place horizontally a
rectangular table of any aspect ratio on a smooth Fenn
graph, and this is even possible for any cyclic
configuration $V=[v_{1},v_{2},v_{3},v_{4}]\subset\mathbb{R}^{3}$. 
In view of the above discussion, it is natural to state the 
following fixed-size conjecture.

\begin{conjecture}\label{conj:cyclic}
  Let $G(f)$ be a Fenn graph over a compact convex domain
  $D\subset\mathbb{R}^{2}$.  Then for any rectangle
  $Q=(v_{1},v_{2},v_{3},v_{4})$ inscribed in the circle
  $S^{1}(r)=\{x\in\mathbb{R}^{2}: \|x\|=r\}$, there exists a
  transformation $T:\mathbb{R}^{2}\to\mathbb{R}^{2}$ such that
  \[
    f(Tv_{1}) = f(Tv_{2}) = f(Tv_{3}) = f(Tv_{4})
    \quad\text{and}\quad T(0)\in D.
  \]
\end{conjecture}

%======================================================================
%  Bibliography
%======================================================================
\bibliographystyle{amsplain}
\bibliography{newtable}

\end{document}